\documentclass[a4paper]{article}

\usepackage{amssymb} \usepackage{amsmath}

\newtheorem{thm}{Theorem}[section]

\newtheorem{prop}[thm]{Proposition}
\newtheorem{defn}[thm]{Definition\rm}
\newtheorem{rem}{\it Remark\/}
\newtheorem{exmp}[thm]{Example}
\newtheorem{conj}{\it Conjecture\/}
\newtheorem{proof}{Proof}

\DeclareMathAlphabet{\mathbbb}{U}{bbold}{m}{n}
\DeclareMathAlphabet{\mathmanual}{U}{manfnt}{m}{n}

\newcommand{\calA}{\mathcal{A}} 
 
 \newcommand{\I}{\mathbbb{1}}

\newcommand{\Sim}{\Sigma}

\newcommand{\Hom}{\mathop{\mathrm{Hom}}}
\newcommand{\End}{\mathop{\mathrm{End}}}

\newcommand{\Q}{\mathbb{Q}}

\newcommand{\tens}{\otimes}

\def\c#1{\mathop{ {\mathcal #1}  }\nolimits}

\def\ch{\mathop{\mathcal Ch}\nolimits}
\def\cch{\mathop{\mathcal CCh}\nolimits}

\def\ch-{\mathop{{\mathcal Ch}_-}\nolimits}
\def\cch-{\mathop{{\mathcal CCh}_-}\nolimits}
\def\ch+{\mathop{{\mathcal Ch}_+}\nolimits}
\def\cch+{\mathop{{\mathcal CCh}_+}\nolimits}

\def\cid#1{\mathop{{\rm Id}_{#1}}\nolimits}

\def\QQ{{\mathbb Q}}

\def\fraz#1 #2 {\frac{#1}{#2}}

\def\h#1{{\kern .07em}^h\kern-.04em{#1}}
\def\a#1{{\kern .07em}{#1}_a}

\let\sem=\bf

\def\To{\longrightarrow}

\begin{document}

\title{Schur finiteness and nilpotency}

\author{Alessio Del Padrone\\delpadro@dima.unige.it\\Carlo Mazza\\carlo@math.ias.edu}

\date{July 5th, 2005}

\maketitle

\begin{abstract}
  Let $\mathcal{A}$ be a $\Q$-linear pseudo-abelian rigid tensor
  category. A notion of finiteness due to Kimura and (independently)
  O'Sullivan guarantees that the ideal of numerically trivial
  endomorphism of an object is nilpotent. We generalize this result
  to special Schur-finite objects.
In particular, in the category of Chow motives,
if $X$ is a smooth projective variety
which satisfies the homological sign conjecture, then
Kimura-finiteness, a special Schur-finiteness,
and the nilpotency
of $CH^{ni}(X^i\times X^i)_{num}$ for all $i$ (where $n=\dim X$)
are all equivalent.
\end{abstract}


%
%

Let $\mathcal{A}$ be a {\sem pseudo-abelian tensor category},
i.e., a ``$\tens$-cat\'egorie rigide sur $F$''
as in \cite[2.2.2]{andremotifs}
in which idempotents split.
We have
$F$-linear {\sem trace} maps ${\rm tr}\colon \End_{\c{A}}(A)\To \End_{\c{A}}(\I)$
compatible with
$\otimes$-functors,
and
$F$-submodules of {\sem numerically trivial morphisms}
$
\c{N}(A_1,A_2):=\{f\in\Hom_{\c{A}}(A_1,A_2)\mid {\rm tr}(f\circ g)=0,\;\; {\rm for\;\;all}\;\;
g\in\Hom_{\c{A}}(A_2,A_1)\}.$
We assume that $F=\End_{\c{A}}(\I)$ and it
contains $\QQ$. If $F$ is a field, $\c{N}$ is the
biggest non trivial $\otimes$-ideal of $\c{A}$,
and so it contains any morphism annihilated
by some $\otimes$-functor.\\

\begin{exmp}\label{examplemotives}\cite[Ch. 4]{andremotifs}
Assume $F$ is a field. 
For any admissible equivalence $\sim$ on algebraic cycles, motives of
smooth projective varieties over a field $k$
with coefficients in $F$
form such a category $\c{A}:=\c{M}_{\sim}(k)_F$.
If $X$ is a variety, we write
$\mathfrak{h}(X)$ for its motive.
For any $f\in\mathrm{End}_{\mathcal{A}}(\mathfrak{h}(X))$, 
$
{\rm tr}(f)={\rm deg}(\Gamma_f\cdot \Delta_X)
$
and therefore
$
\c{N}(\mathfrak{h}(X))=\c{Z}^{{\rm dim}(X)}_\sim(X\times X)_{F,\rm num}
$
(numerically trivial
correspondences of degree zero).
If $\sim$ is finer than homological equivalence
then any Weil cohomology $H$ factors through a $\tens$-functor on 
$\mathcal{A}$, and
$
{\rm tr}(f)=
\sum_{j} (-1)^j{\rm Tr}(f|H^j(X))
$
by the Lefschetz
formula.\\
\end{exmp}

Recall that the partitions $\lambda$ of an integer $n$
give a complete set of mutually orthogonal central idempotents
$ d_\lambda:=\frac{\dim V_\lambda}{n!}\sum_{\sigma\in \Sigma_n}
\chi_\lambda(\sigma)\sigma$
in the group algebra $\Q\Sim_n$ (see \cite{FH}).
We define an endofunctor on $\c{A}$ by setting $S_\lambda(A)=d_\lambda(A^{\tens n})$. This is a multiple of the classical Schur functor corresponding to $\lambda$. In particular, we define $\mathrm{Sym}^n(A)=S_{(n)}(A)$
and $\Lambda^n(A)=S_{(1^n)}(A)$. 
The following definitions are directly inspired by
\cite{delschur} and \cite{kimura} (see \cite{yveskahn},
\cite{gp2}, and \cite{maz} for further reference).\\

\begin{defn}
An object $A$ of $\calA$ is \textbf{Schur-finite} if there
is a partition $\lambda$ such that
$S_\lambda(A)=0$.
If $S_\lambda(A)=0$ with 
$\lambda$ of the form $(n)$ (respectively, $\lambda=(1^n)$) then $A$ is called
\textbf{odd} (respectively, \textbf{even}).
We say that $A$ is \textbf{Kimura-finite}
if $A=A_+\oplus A_-$ with $A_+$ even and $A_-$ odd.\\
%
\end{defn}

Every Kimura-finite object is Schur-finite, but the converse 
fails, for example, in the category of super-representations of $GL(p|q)$.
In \cite[7.5]{kimura} and \cite[9.1.14]{yveskahn}
it was proven that if $A$ is a Kimura-finite object then
the ideal $\mathcal{N}(A)$ is nilpotent.

In the case of example \ref{examplemotives}, an interesting
consequence of the nilpotence of $\mathcal{N}(M)$ is
that a summand $N$ of $M$ is zero if and only if its
cohomology is zero (the idempotent defining $N$ must then be nilpotent).
The nilpotency was used in \cite[Theorem 7]{gp2} to
show the equivalence of Bloch's conjecture for a smooth projective suface $X$ with $p_g=0$
and the Kimura-finiteness of the motive of $X$, improving \cite[7.7]{kimura}.

Albeit in general Schur-finiteness
is not sufficient to get the nilpotency of $\mathcal{N}(A)$
(see \cite[10.1.1]{yveskahn}),
we will identify additional conditions which imply the nilpotency.
In the category of motives we will show that for a motive
which is Kimura-finite modulo homological equivalence, the Kimura-finiteness
modulo rational equivalence is equivalent to the Schur-finiteness
for a particular rectangle.

\section{A technical result}

\begin{thm}\label{mainthm}
Suppose that $S_\lambda(A)=0$ for a partition
$\lambda$ of $n\geq 2$ with $a_\lambda$ rows and $b_\lambda$ columns. Let
$s:=a_\lambda+b_\lambda-1$ be the length of its biggest hook $\nu$, and $r:=n-s$.
Assume that either $\lambda$ is a hook or
that there is a $g\in\End_{\c{A}}(A)$ with trace
$t:=\mathrm{tr}(g)=\dots=\mathrm{tr}(g^{\circ r})$, and
$t\not\in  \{-(b_\lambda-2),\dots,a_\lambda-2\}$. Then
$f^{\circ (s-1)}=0$ for each $f\in\mathcal{N}(A)$, and so 
$\mathcal{N}(A)$ is nilpotent.
\end{thm}
\begin{proof}
The last statement follows from
\cite[7.2.8]{yveskahn}:
$\mathcal{N}(A)^{2^{s-1}-1}=0$.


For $\sigma\in \Sigma_n$, we index
the corresponding decomposition
of $\{1,\ldots,n\}$ 
into disjoint cycles $\gamma_1,\ldots,\gamma_n$ so that
the support of $\gamma_1$ contains $1$; moreover
we define $l_i$ to be the order
of the cycle $\gamma_i$, and $L=L(\sigma):={\rm max}_i\{l_i \}$ to be the
maximum length of the cycles of $\sigma$.  

As $S_\lambda(A)=0$ we have
$\sum_\sigma  \chi_\lambda(\sigma)\cdot
\sigma \circ f_1\tens \cdots \tens f_n=0
$
for any
$f_1,\dots,f_n\in \End_{\c{A}}(A)$.
By the Murnaghan-Nakayama rule (see \cite[Problem 4.45]{FH})
$\chi_\lambda(\sigma)=0$ if $L(\sigma)>s$.
Hence \cite[7.2.6]{yveskahn}
with $A_1=\cdots=A_n=A$, gives that in $\End_{\c{A}}(A)$
\[
\sum_{\sigma\in \Sigma_n:\;L(\sigma)\leq s} \chi_\lambda(\sigma)\cdot
 t_\sigma\cdot f_{\gamma_1} = 0,
\]
where 
$f_{\gamma_1}:=
f_{\gamma_1^{l_1-1}(1)}\circ \cdots \circ f_{\gamma_1(1)}\circ f_1$,
$t_\sigma:=\prod_{j=2}^q t_{\sigma,j}$,
and 
$
t_{\sigma,j}:= {\rm tr}(f_{{\gamma_j}^{l_j-1}(k_j)}\circ \cdots \circ
f_{\gamma_j(k_j)}\circ f_{k_j})
$
with $k_j$ any element in the support of $\gamma_j$
(if $l_1=n$, i.e. $q=1$, then $t_\sigma=1$).

Set 
$f_1:=\cid{A}$ and $f_2=\dots=f_s:=f$
(still no restrictions on $f_{s+1},\ldots, f_n$).
%
If ${\rm Supp}(\gamma_1)\subsetneq \{1,\dots, s\}$, not all
of the $f$'s are in the composition
$f_{\gamma_1}$, hence at least one of them must appear in a trace
 ${\rm tr}(f_{{\gamma_j}^{l_j-1}(k_j)}\circ \cdots \circ
f_{\gamma_j(k_j)}\circ f_{k_j})$. But $f$ is numerically trivial, so
$t_\sigma=0$
for any such $\sigma$, and
$$
0=\sum_{\sigma\in \Sigma_n:\;{\rm Supp}(\gamma_1)=\{1,\dots, s\}}
\chi_\lambda(\sigma)\cdot
t_\sigma \cdot f_{\gamma_1}
=
\left(\sum_{\sigma\in \Sigma_n:\;{\rm Supp}(\gamma_1)=\{1,\dots, s\}}
\chi_\lambda(\sigma)\cdot
t_\sigma\right)f^{\circ(s-1)}=x\cdot f^{\circ(s-1)},
$$
where
$
x:=\sum_{\sigma\in \Sigma_n:\;{\rm Supp}(\gamma_1)=\{1,\dots, s\}}
\chi_\lambda(\sigma)\cdot
t_\sigma\in F$.
It is enough to show $x\not=0$ for some choice
of the $f_i$'s.

If
$r=0$ then $\lambda=\nu=(n-j,1^j)$ is itself a hook,
$t_\sigma=1$ for any
$\sigma$ with $l_1=n$ and by
\cite[Exercise 4.16]{FH} $x$ is just $(n-1)!(-1)^j\not=0$,
hence $\c{N}(A)$ is nilpotent.

If $\lambda$ is not a hook let $\delta:=\lambda\setminus\nu$.
The element $x\in F$
is a sum over
$\sigma=\gamma_1\circ\sigma^\prime$
such that $\gamma_1$ is an $s$-cycle of $\{1,\dots,s\}$ and
$\sigma^\prime$ is a permutation of $\{s+1,\dots,n\}$, so
by Murnaghan-Nakayama
$
\chi_{\lambda}(\sigma)=
\chi_{\lambda\setminus \nu}(\sigma^\prime),
$
and $
x=(-1)^{a_\delta-1}\;
|\{s-{\rm cycles\;\, of\;\,} \Sigma_n\}|
\sum_{\sigma^\prime\in \Sigma_r}
\chi_{\delta}(\sigma^\prime)\cdot t_\sigma$.
Thus we are reduced to study elements of the form
\[
y(\delta;g_1,\dots,g_{r}):=\sum_{\sigma\in \Sigma_{r}}
\chi_{\delta}(\sigma)\cdot
\prod_{j=1}^q t_{\sigma,j},
\]
where we can choose freely $g_1,\dots,g_{r}\in\End_{\c{A}}(A)$.

Take $g\in\End_{\c{A}}(A)$ as in the hypothesis, then
$y(\delta;g,\dots,g)=
\sum_{\sigma\in \Sigma_{r}}
\chi_{\delta}(\sigma)\cdot t^{|{\rm cycles\;\; of \;\,} \sigma|}
$
is the polynomial in $t={\rm tr}(g)$ called the 
{\sem content polynomial}
of $\delta$. It decomposes as $y(\delta;g)=
\chi_{\delta}(\cid{\Sigma_{r}})\cdot\prod_{(i,j)\in \delta}(t+j-i)$,
then $y(\delta;g)=0$
if and only if
${\rm tr}(g)\in \{-(b_\delta-1),\dots,a_\delta-1\}
\subseteq
\{-(b_\lambda-2),\dots,a_\lambda-2\}
$.
By hypothesis, there is a $g$ such that $y(\delta;g)\not=0$, which
implies that $x\not =0$, which in turn implies that $f$ is nilpotent.
Hence the theorem is proven.
\end{proof}

\begin{rem}[B. Kahn]
The existence of a 
$g\in\End_{\c{A}}(A)$ with ${\rm tr}(g)\neq 0$ is not enough
to ensure the nilpotency of $\c{N}(A)$ with $A$ Schur-finite. 
In \cite[10.1.1]{yveskahn} it is exhibited a
non-zero Schur-finite object $A^\prime$ with $\c{N}(A^\prime)=\End_{\c{A}}(A^\prime)$:
it suffices to look at $A:=A^\prime\oplus\I^n$.\\
\end{rem}
%
%
%
%
\begin{conj}\label{con}
From numerical evidence (\cite{GAP4}) we conjecture a stronger version of
Theorem \ref{mainthm}. Let
$A$ be an object with two
endomorphisms $\pi_1$ and $\pi_2$ 
such that $a:=\mathrm{tr}(\pi_1)=\mathrm{tr}(\pi_1^{\circ i})$
for all $i$,
$b:=\mathrm{tr}(\pi_2)=\mathrm{tr}(\pi_2^{\circ j})$ for all $j$,
and $\mathrm{tr}(\pi_1^{\circ i}\circ \pi_2^{\circ j})=0$ for all $i$ and $j$.
If $S_\lambda(A)=0$ where $\lambda\not\supset (b+2)^{a+2}$, then
$y(\lambda\setminus\nu;\alpha_1\pi_1+\alpha_2\pi_2)\not =0$ (as a polynomial
in $\alpha_1$ and $\alpha_2$) and hence $\mathcal{N}(A)$ is nilpotent.
\end{conj}


\section{Motives and nilpotency}

Let now $\c{A}$ be the category of Chow motives $\c{M}_{rat}(k)_\QQ$  (example \ref{examplemotives}),
let $H$ be any Weil cohomology, and
let $X$ be a smooth
projective variety.
The cohomology $H(X)$
is a super vector space
of dimension $(d_{ev},d_{odd})$, and
we set $\lambda_{H(X)}:=((d_{odd}+1)^{d_{ev}+1})$
(the rectangle with $d_{odd}+1$ columns and  $d_{ev}+1$ rows).
By \cite[1.9]{delschur}, $S_\lambda(H(X))\not=0$ if and only if
$\lambda\not\supset\lambda_{H(X)}$.
Hence, $S_\lambda(\mathfrak{h}(X))\not=0$ if $\lambda\not\supset
\lambda_{H(X)}$. So $S_\lambda(\mathfrak{h}(X))=0$ implies that $\lambda\supset
\lambda_{H(X)}$.



Recall the ``homological sign conjecture''
(due to Jannsen, see \cite[5.1.3]{andremotifs}): we say that
$X$ satisfies the conjecture $C^+(X)$ if
the projections on the even
and the odd part of the cohomology are algebraic.
This conjecture is stable under products, and it holds true,
with respect to classical cohomologies,
for abelian varieties and
smooth projective varieties of dimension at most two.
It can be shown that $C^+(X)$ is equivalent to the Kimura-finiteness 
of the motive of $X$ modulo
homological equivalence.\\

\begin{prop}\label{mainprop}
Let $X$ be a smooth projective variety,
and let $\lambda$ be a partition with at most
$d_{ev}+1$ rows or $d_{odd}+1$ columns.
If $S_{\lambda}(\mathfrak{h}(X))=0$
(and hence $\lambda\supset\lambda_{H(X)}$)
and $C^+(X)$ holds,
then $\c{N}(\mathfrak{h}(X))$ is nilpotent.
Moreover, if $X$ is a surface with $p_g=0$, Bloch's conjecture holds for $X$.
\end{prop}

\begin{proof}
By $C^+(X)$ there are two cycles $\pi_+$ and $\pi_-$ inducing
the projections on the even and odd cohomology. Then
$d_{ev}=\mathrm{tr}(\pi_+)=\mathrm{tr}(\pi_+^{\circ i})$ for all $i$,
and $-d_{odd}=\mathrm{tr}(\pi_-)=\mathrm{tr}(\pi_-^{\circ j})$ for all $j$.
Then either $\pi_+$ or $\pi_-$ satisfies the condition of 
Theorem \ref{mainthm}, and therefore $\c{N}(\mathfrak{h}(X))$ is nilpotent.
Bloch's conjecture is now a formal consequence of \cite[7.6 and 7.7]{kimura}.
\end{proof}

\begin{thm}
Let $X$ be a smooth projective variety.
Under $C^+(X)$ the following are equivalent:
\[ 1)\ \text{$\mathfrak{h}(X)$ is Kimura-finite;}\qquad
2)\ \text{$S_{\lambda_{H(X)}}(\mathfrak{h}(X))=0$;}\qquad
3)\ \text{$\mathcal{N}(\mathfrak{h}(X^n))$ is nilpotent for all $n\geq 1$.}
\]
\end{thm}

\begin{proof}
It is easy to show that $1\Rightarrow 2$. 
For $3\Rightarrow 1$ we proceed as follows. As $C^+(X)$ holds and $\mathcal{N}(\mathfrak{h}(X))$ is nilpotent,
then there exist two motives
$X_+$ and $X_-$ whose cohomologies are exactly the even and the odd
part of $H(X)$. It is now easy to prove
that $\mathfrak{h}(X)=M_+\oplus M_-$ with $M_+$ even and
$M_-$ odd
because it will be enough to check it in cohomology.
We need to verify $2\Rightarrow 3$.
Assume that $S_{\lambda_{H(X)}} (\mathfrak{h}(X))=0$. 
From the proof of
\cite[Cor. 1.13]{delschur}, we find that $S_{\lambda_{H(X^n)}}(\mathfrak{h}(X^n))=
S_{\lambda_{H(X^n)}}(\mathfrak{h}(X)^{\tens n})=0$.
Since $C^+(X^n)$ holds true,
Proposition \ref{mainprop} gives that $\mathcal{N}(\mathfrak{h}(X^n))$ is nilpotent.
\end{proof}

If Conjecture \ref{con} is true, then Bloch's conjecture holds
for any smooth projective surface $X$ with $p_g=0$
such that $S_\lambda(\mathfrak{h}(X))=0$ for $\lambda\not\supset
(d_{odd}(X)+2)^{d_{ev}(X)+2}$.





\section*{Acknowledgments}

This article was inspired by the preprint \cite{Gul2} by V. Guletski{\u\i}.
Both authors would like to thank the Institut de Math\'ematiques de Jussieu
for hospitality while part of this manuscript was written.
For financial support, 
the first author would like to thank the Fondazione Carige 
and the second author would like to thank  
Sergio Serapioni (Honorary President of the Societ\`a 
Trentina Lieviti, Trento, Italy).

\bibliographystyle{amsalpha}
\bibliography{paper}
\end{document}